\documentclass[11pt]{article}

\usepackage{iftex}
\ifpdftex
    \usepackage{CJKutf8}
\else
    \usepackage[scheme=plain]{ctex}
    
    \newenvironment{CJK*}[2]{}{}
\fi

\usepackage{amsmath,amsfonts,amsthm,amssymb}
\usepackage{mathtools,bm,mathrsfs}
\usepackage{graphicx}
\usepackage[colorlinks=true, allcolors=blue]{hyperref}
\usepackage{thmtools}
\usepackage[capitalise]{cleveref}
\usepackage{standalone}
\usepackage{tikz}
\usetikzlibrary{shapes.geometric}
\usetikzlibrary{shapes.misc}
\usepackage{indentfirst}
\usepackage{authblk}
\usepackage{diagbox}
\usepackage{changes}
\usepackage{enumitem}
\usepackage{bbm}
\usepackage{orcidlink}

\newtheorem{theorem}{Theorem}[section]
\newtheorem{lemma}[theorem]{Lemma}
\newtheorem{corollary}[theorem]{Corollary}

\newtheorem{case}{Case}

\newtheorem{definition}[theorem]{Definition}

\newtheorem*{claim*}{Claim}

\newcommand{\Ex}{\mathbb{E}}

\newcommand{\cR}{\mathcal{R}}

\newcommand{\smallo}{o}

\newcommand{\LLY}{\mathrm{LLY}}

\newcommand\ceil[1]{\left\lceil #1\right\rceil}

\DeclarePairedDelimiter{\card}{\lvert}{\rvert}
\DeclarePairedDelimiter{\set}{\lbrace}{\rbrace}

\DeclarePairedDelimiter{\paren}{\lparen}{\rparen}

\newcommand\floor[1]{\left\lfloor #1\right\rfloor}
\title{Bipartite graphs, random graphs, and Lin--Lu--Yau curvature}

\author{Huiqiu Lin\footnote{email: huiqiulin@126.com}~\orcidlink{0000-0002-6072-4647}}

\author{Zhe You\footnote{email: y30231280@mail.ecust.edu.cn}~\orcidlink{0009-0006-9769-5372}}

\author{Da Zhao\footnote{email: zhaoda@ecust.edu.cn}~\orcidlink{0000-0002-9582-0778}}

\affil{School of Mathematics, East China University of Science and Technology}
\date{}

\begin{document}
\maketitle

\begin{abstract}
 Let $G = (X, Y; E)$ be a bipartite graph with parts $X$ and $Y$ where $|X|=m$ and $|Y|=n$. We show that every bipartite graph with more than $mn - D(m,n)$ edges has positive Lin--Lu--Yau curvature, where $D(m,n)=m-2+\ceil{\frac {n}{2}} \text{ if $n\geq 2m$}, \mbox{and} \
         n-1 \text{ if $m\leq n< 2m$}.$
We also show that every bipartite graph of order $m+n$ with $m \geq n$ and minimum degree at least $\min\{n, \floor{\frac{m+n}{3}}+1\}$ has positive Lin--Lu--Yau curvature.
Both bounds are sharp. 
Meanwhile probabilistically we can relax the edge density conditions in above results. 
It is shown that relatively dense random bipartite graph is positively curved. 
All of our proofs are based on a new formula for Lin--Lu--Yau curvature of bipartite graphs.
\end{abstract}



\section{Introduction}

We consider simple finite graphs in this paper. 
Let $G = (V, E)$ be a simple graph. 
We use the notation $G = (X, Y; E)$ for a bipartite graph with parts $X$ and $Y$.

Lin, Lu and Yau~\cite{LLY11} introduced the notion of Lin--Lu--Yau curvature $\kappa_{\LLY}(x, y)$, abbreviated as LLY curvature, based on Ollivier curvature~\cite{ollivier2009RicciCurvatureMarkov}. 
It is defined on any two distinct vertices on a graph. 
We call $\kappa_{\LLY}(x,y)$ the Lin--Lu--Yau curvature of an edge whenever $x$ and $y$ are adjacent. 
The Lin--Lu--Yau curvature is defined via optimal transport between neighborhood of two vertices. 
The precise definition is formulated via the Wasserstein distance between probability measures, and we postpone it to the next section. 
A graph $G$ is positively curved if the LLY curvature of each of its edges is positive. 

The {\it connectivity} $k(G)$ of a non-complete graph $G$ is defined as the minimum number of vertices whose removal disconnects the graph~\cite{bondy2008GraphTheory,diestel2025GraphTheory}. 
For consistency, the connectivity of a complete graph $K_n$ with $n$ vertices is defined to be $n-1$.
The {\it edge-connectivity} $k'(G)$ of a graph $G$ with at least two vertices is defined as the minimum number of edges whose removal disconnects the graph.
For consistency, the edge-connectivity of a graph with only a single vertex is defined to be $0$.

For a graph $G$ with $n$ vertices, Hehl~\cite{Hehl25} proved that if $\delta(G)\geq \frac{2n}{3}-1$, then $\kappa_{\LLY}(G)\geq 0$. 
Moreover, the bound $\frac{2n}{3}-1$ is optimal.
Similarly, Chen, Liu and You~\cite{CLY} proved that the Lin-Lu-Yau curvature of $G$ is positive whenever the connectivity $k(G)$ of a graph $G$ with $n$ vertices is at least $\frac{n-1}{2}$. 
the bound $\frac{n-1}{2}$ is also optimal.
In this paper, we consider the similar questions in the bipartite setting.

\begin{theorem}\label{thm:main}
Let $G=(X,Y;E)$ be a bipartite graph with $|X|=m\ge n=|Y|$.  If
\[
   \delta(G)\ge \cR(m,n)\coloneqq\min\left\{n,\left\lfloor\frac{m+n}{3}\right\rfloor+1\right\},
\]
then $G$ has positive Lin--Lu--Yau curvature.
\end{theorem}
Moreover, this threshold is sharp.
For $n\geq 2$, there exists a graph $G_0$ with minimum degree $\cR(m,n)-1$ that has an edge with non-positive curvature, as shown in Section~\ref{proof of main}.
The graph $G_0$ we construct satisfy $k(G_0)=k'(G_0)=\delta(G_0)$.   
Recall a classical estimate of connectivity of a graph $G$ due to Whitney~\cite{Whi32} reads as
\begin{equation}\label{whitney}
    k(G)\leq k'(G)\leq \delta(G).
\end{equation}
Thus, we have the following corollary.

\begin{corollary}\label{connectivity}
    Let $G=(X,Y;E)$ be a bipartite graph with $|X|=m\ge n=|Y|$.  
    If $ k(G)\ge \cR(m,n)$ (or $k'(G)\ge \cR(m,n)$), 
then $G$ has positive Lin--Lu--Yau curvature, and the threshold is sharp.
\end{corollary}

In another paper by Chen, Liu and You~\cite{chen2026ExtremalTheoremPositive}, they showed that every graph of order $n \geq 8$ with more than $\frac{n^2 - 3n}{2} - \ceil{\frac{n}{2}} + 2$ edges has positive Ollivier/Lin--Lu--Yau curvature and the threshold is optimal. 
A closely related criterion for positive Lin--Lu--Yau curvature, formulated in terms of forbidden subgraphs in the complement, was obtained in~\cite{CLY26}.
Our next result of this paper provides an extremal result on bipartite graphs.

\begin{theorem}\label{main}
 For integers $m,n$, set 
 \begin{align*}
     D(m,n)&\coloneqq 
     \begin{cases}
         m-2+\ceil{\frac {n}{2}}, \text{ if $n\geq 2m$};\\
         n-1, \text{ if $m\leq n< 2m$}.
     \end{cases}
\end{align*}
Then for any $2\leq m \leq n$, any bipartite graph $G=(X,Y;E)$ with $|X|=m$ and $|Y|=n$, and more than $mn-D(m,n)$ edges has positive Lin--Lu--Yau curvature.
Moreover, there exists a bipartite graph with $mn-D(m,n)$ edges contains an edge of non-positive Lin--Lu--Yau curvature.
\end{theorem}
For connectivity, diameter, eigenvalues, and other combinatorial properties of graphs
with positive Lin--Lu--Yau curvature, see
\cite{CLY,CKKLMP20,Hehl-regular,MW19} and the references therein.
One may also consider the similar problems for other notions of discrete curvature~\cite{NR17}.

The above results put strong restriction on number of edges. 
In fact, probabilistically the edge density can be relaxed. 
We show that relatively dense random bipartite graphs are positively curved. 
A random bipartite graph $B(m, n, p)$ is obtained with two parts each of size $m$ and $n$ respectively, and each edge between two parts occurs with probability $p$. 

\begin{theorem}\label{thm:random_bipartite_positive}
    Let $G$ be a random bipartite graph $B(m,n,p)$ with parts $X$ and $Y$.  
    Denote by $N = m + n$ and $s = \min(m,n)$. 
    Suppose $p^2 s \gg \log N$. 
    Then
    \begin{align*}
         \lim_{N \to \infty} \Pr\paren*{G \text{ is positively curved}} = 1.
    \end{align*}
\end{theorem}



The paper is organized as follows. 
In~\cref{sec:preliminary}, we prepare notations and tools for the proofs. 
In~\cref{proof of main}, we prove the main theorem, ~\cref{thm:main}. 
In~\cref{sec:extremal}, we consider the extremal bipartite graphs.
In~\cref{sec:random_bipartite}, we prove~\cref{thm:random_bipartite_positive}.

\section{Preliminaries}\label{sec:preliminary}

Throughout the paper, we use the following notation. 
Let $G=(V,E)$ be a simple finite graph.
For any $x\in V$ and $A\subseteq V$, let $N_A(x)$ be the set of neighbors of $x$ in $A$ and let $d_x\coloneqq|N_V(x)|$ be its degree. 
For simplicity, we abbreviate $N_V(x)$ as $N(x)$.
A vertex $y\ne x$ which is not adjacent to $x$ is called a non-neighbor of $x$.
For any $S\subseteq V$, define $N_A(S)\coloneqq\cup_{v\in S}N_A(v)$.
For any two vertices $x$ and $y$ in $V(G)$, we denote the distance between them by $\rho(x,y)$. 
For two fixed vertices $x$ and $y$  in $V(G)$, we set $A_{xy}\coloneqq N(x)\cap N(y)$, $N^{(xy)}_x\coloneqq N(x)\backslash (A_{xy}\cup \{ y \})$, and $N^{(xy)}_y\coloneqq N(y)\backslash (A_{xy}\cup \{ x \})$.
For simplicity, we write $N_x$ for $N^{(xy)}_x$ and write $N_y$ for $N^{(xy)}_y$ hereafter.

\subsection{Lin--Lu--Yau curvature}
Before introducing the Lin--Lu--Yau curvature, we first recall the definition of the Wasserstein distance.
\begin{definition}[$L^1$-Wasserstein distance]
     Let $G=(V,E)$ be a locally finite graph, $\mu_1$ and $\mu_2$ be two probability measures on $G$. 
     The {\it $L^1$-Wasserstein distance} $W(\mu_1, \mu_2)$ between $\mu_1$ and $\mu_2$ is defined as
     \begin{align}\label{defi}
         W(\mu_1,\mu_2)=\inf_{\pi}\sum_{u\in V}\sum_{v\in V}\rho(u,v)\pi(u,v),
     \end{align}
     where the infimum is taken over all the mappings $\pi: V\times V\to [0,1]$ satisfying
     $$\mu_1(u)=\sum\limits_{v\in V}\pi(u,v) \text{ for any}\ u\in V$$
     and
     $$\mu_2(v)=\sum\limits_{u\in V}\pi(u,v) \text{ for any}\ v\in V.$$ 
     Such a mapping is called a {\it transport plan} from $\mu_1$ to $\mu_2$. 
     A transport plan that attains the infimum in \eqref{defi} is called {\it optimal}.
\end{definition}
 Here, for a given idleness parameter $p\in [0,1]$, we consider the particular measure $\mu_x^p$ around a vertex $x\in V$ defined as follows:
    \[\mu_x^p(y)=\left\{
                    \begin{array}{ll}
                      p, & \hbox{if $y=x$;} \\
                      \frac{1-p}{d_x}, & \hbox{if $xy\in E$;} \\
                      0, & \hbox{otherwise.}
                    \end{array}
                  \right.
     \]
     
   Based on the probability measure above, two kinds of Ricci curvature on graphs are defined as follows.
\begin{definition}[$p$-Ollivier curvature and Lin--Lu--Yau curvature] 
Let $G=(V,E)$ be a locally finite graph. 
For any two distinct vertices $x,y$ in $G$, the {\it $p$-Ollivier curvature} $\kappa_p(x,y)$, $p\in [0,1]$, is defined as
     \[\kappa_p(x,y)=1-\frac{W(\mu_x^p,\mu_y^p)}{\rho(x,y)}.\]
The {\it Lin--Lu--Yau curvature} $\kappa_{\LLY}(x,y)$ is defined as
     \[\kappa_{\LLY}(x,y)=\lim_{p\to 1}\frac{\kappa_p(x,y)}{1-p}.\]
\end{definition}
It is worth noting that the ratio $\kappa_p(x,y)/(1-p)$ is constant when $p$ is large enough. 
In fact, it was proved in \cite{BCLMP18} that
\begin{equation}\label{eq:bourne}
  \kappa_{\LLY}(x,y)=\frac{\kappa_p(x,y)}{1-p} \,\,\text{for any $p \in \left[\frac{1}{\max\{d_x,d_y\}+1},1\right)$}.
\end{equation}
In particular we have 
\[
  \kappa_{\LLY}(x,y)>0
  \quad
  \Longleftrightarrow
  \quad
  \kappa_{1/2}(x,y)>0.
\]
Therefore, all of our results are also true for $p$-Ollivier with $\frac{1}{2}\leq p< 1$.

For any locally finite graph $G$, the  normalized graph Laplacian $\Delta$ is defined as $$\Delta f(x)\coloneqq\frac{1}{d_x} \sum_{y: xy\in E(G)}(f(y)-f(x)), \text{ for any $f: V(G)\to \mathbb{R}$ and any $x\in V(G)$}.$$
 A function $f$ on $V(G)$ is called $1$-Lipschitz if  $|f(x)-f(y)|\leq \rho(x,y)$ holds for any $x,y\in V(G)$, and we denote the set of $1$-Lipschitz functions by $Lip(1)$.
 There is another limit-free formulation of Lin--Lu--Yau curvature, which was given by M\"{u}nch and Wojciechowski~\cite{MW19}. 
 \begin{theorem}[Curvature via Laplacian {\cite[Corollary 2.2]{MW19}}]\label{Curvature via the Laplacian}
     Let $G$ be a locally finite graph and let $xy$ be an edge. Then
     $$\kappa_{\LLY}(x, y)=\inf _{\substack{f:N(x)\cup N(y)\to \mathbb{Z}\\f \in Lip(1) \\ f(y)-f(x)=1}} \left(\Delta f(x)-\Delta f(y)\right).$$
 \end{theorem}

We present some useful Lemmas in our proof as follows, which can be obtained based on~\cref{Curvature via the Laplacian}.

\begin{lemma}[{\cite[Lemma 1]{LLY14}}]\label{no C3 C4 C5}
    Let $G$ be a locally finite simple graph.
    Suppose that an edge $xy$ in the graph $G$ is not in any $C_3$, $C_4$, or $C_5$. 
    Then
$$\kappa_{\LLY}(x,y)=\frac{2}{d_x}+\frac{2}{d_y}-2.$$
\end{lemma}

\begin{lemma}[{\cite[Corollary 2]{li2024ricci}}]\label{LLYupperbound}
    For any edge $xy\in E(G)$, assume $d_x \leq d_y$. 
    We have
    $$\kappa_{\LLY}(x,y)\leq \frac{|N(x)\cap N(y)|+2}{d_y}.$$
\end{lemma}

\subsection{Matching}
A {\it matching} in a graph is a set of pairwise non-adjacent edges. 
A matching $M$ is called {\it a matching of} $U\subseteq V$ if every vertex in $U$ is incident with an edge in $M$.
The following Lemma is a direct corollary of Hall's Marriage Theorem~\cite[Theorem 2.1.2]{diestel2025GraphTheory}).


\begin{lemma}[{\cite[Corollary 2.7]{CLY}}]\label{Hall-c}
    Let $H=(V_1,V_2;E)$ be a bipartite graph.
    Suppose that there is a non-negative integer $c$ such that
    $$|N_{V_2}(A)|\geq |A|-c \,\,\text{for all}\,\,A\subseteq V_1.$$
    Then $H$ contains a matching of size at least $|V_1|-c$.
\end{lemma}

\section{Lower bound of minimum degree}\label{proof of main}

In this section, we prove~\cref{thm:main}. 
We first prove a curvature formula for  bipartite graphs as follows.
\begin{lemma}\label{lem:hall-formula}
Let $G$ be a bipartite graph.
For every edge $xy\in E(G)$,  define $H_{xy}\coloneqq G[N_x,N_y]$.
Then
\begin{align*}
\kappa_{\LLY}(x,y)
 =2\min_{U\subseteq N_y}
 \left(
 \frac{|N_{H_{xy}}(U)|+1}{d_x}-\frac{|U|}{d_y}
 \right).
\end{align*}
\end{lemma}

\begin{proof}
For the upper bound, we consider the function $f_U:N(x)\cup N(y)\rightarrow\mathbb{Z}$ for any $U\subseteq N_y$ given by
$$f_U(z)= 
\begin{cases}
-1, & \text { if } z\in N_x\setminus N_{H_{xy}}(U);\\
0, & \text { if } z\in \{x\} \cup N_y\setminus U; \\
1, & \text { if } z\in \{y\}\cup N_{H_{xy}}(U)\\
2, & \text { if } z\in U.
\end{cases}$$
Then, ~\cref{Curvature via the Laplacian} yields
\[\kappa_{\LLY}(x,y)\leq 2
 \left(
 \frac{|N_{H_{xy}}(U)|+1}{d_x}-\frac{|U|}{d_y}
 \right).\]
 The upper bound holds by taking the minimum over all $U\subseteq N_y$.

 Now we consider the lower bound.
Set $D\coloneqq\max\limits_{U\subseteq N_y}\{d_x|U|-d_y|N_{H_{xy}}(U)|\}$.
If $U=\emptyset$, then $d_x|U|-d_y|N_{H_{xy}}(U)|=0$.
Hence $D\geq 0$.
We construct a transport graph $H'_{xy}$ of $H_{xy}$ as follows.
Duplicate each vertex in \(N_x\) into \(d_y\) vertices, and duplicate each vertex in \(N_y\) into \(d_x\) vertices.
The adjacency relations of the duplicated vertices remain the same as those of their original vertices in \(H_{xy}\).
Now we obtain a bipartite graph \(H'_{xy}\) with the partition $ (H_x, H_y$), where $H_x$ and $H_y$ respectively correspond to the copies of vertices in $N_x$ and $N_y$.

Note that $D+ d_y|N_{H_{xy}}(U)|\geq d_x |U|$ for every subset $U\subseteq N_y$.
We have $|N_{H_x}(S)|\geq |S|-D$ for every subset $S\subseteq H_y$.
By~\cref{Hall-c}, $H'_{xy}$ contains a maximum matching of size at least $|H_y|-D$, denoted by $M$.
Denote the set of vertices in $H_y$ which are not matched in $M$ by $P$ and denote the set of vertices in $H_x$ which are not matched in $M$ by $Q$.
Note that there is no edge between $P$ and $Q$; otherwise we can find a larger matching.
For convenience, denote any copy of $z$ in $H_{xy}$ by $z_0$.
Let $c_{uv}$ be the number of edges in $M$ between the copies of $u\in N_x$  and $v\in N_y$.
\begin{case}
    $d_x\geq d_y$.
\end{case}

Since $|H_x|\geq |H_y|$, there exists an injection $\phi:P\rightarrow Q$. 
By the definition of $\phi$, the original vertex $v$ of $v_0\in P$ has distance $3$ from the original vertex $u$ of $\phi(v_0)$.
Let $h_{uv}$ be the number of pairs $(\phi(v_0),v_0)$ for $v_0\in P$.
Hence $c_{uv}\cdot h_{uv}=0$.
For any $\frac{1}{1+d_y}\leq p<1$, let $\pi_p: V\times V\to [0,1]$ be the map defined as follows:
    \begin{center}
    $\pi_p(u,v)=\begin{cases}
    p-\frac{1-p}{d_y}, &{\rm if}\ u=x, v=y;\\
    \frac{1-p}{d_y}, &{\rm if}\ u=v=x;\\
    \frac{1-p}{d_x}, &{\rm if}\ u=v=y;\\
    (c_{uv}+h_{uv})\cdot\frac{1-p}{d_xd_y}, &{\rm if}\ u\in N_x,v\in N_y;\\
    (d_y-\sum\limits_{w\in N_y} (c_{uw}+h_{uw}))\cdot\frac{1-p}{d_xd_y}, &{\rm if}\ u\in N_x,v=y;\\
    0, &{\rm otherwise}.
    \end{cases}$
    \end{center}
Clearly, $\pi_p$ is a transport plan.
Recall that $|M|+|P|=|H_y|=(d_y-1)d_x$.
Therefore,
\begin{align*}
    W(\mu_x^p,\mu_y^p)&\leq  \sum_{u,v\in V}\rho(u,v)\pi_p(u,v)\\
   & =\pi_p(x,y)+\sum_{\substack{u\in N_x,v\in N_y\\ c_{uv}>0}}\pi_p(u,v)+2\sum_{u\in N_x}\pi_p(u,y)+3\sum_{\substack{u\in N_x, v\in N_y\\h_{uv}>0}}\pi_p(u,v)\\
   &=\left(p-\frac{1-p}{d_y}\right)+\left(|M|\cdot \frac{1-p}{d_xd_y}\right)+2\left(\frac{1-p}{d_y}-\frac{1-p}{d_x}\right)+3\left(|P|\cdot\frac{1-p}{d_xd_y}\right)\\
   &=p+\frac{1-p}{d_y}-2\cdot\frac{1-p}{d_x}+(d_y-1)d_x\cdot\frac{1-p}{d_xd_y}+2|P|\cdot\frac{1-p}{d_xd_y}\\
   &\leq p+\frac{1-p}{d_y}-2\cdot\frac{1-p}{d_x}+(d_y-1)d_x\cdot\frac{1-p}{d_xd_y}+2D\cdot\frac{1-p}{d_xd_y}\\
   &=1-2\cdot\frac{1-p}{d_x}+2D\cdot\frac{1-p}{d_xd_y}.
\end{align*}
It follows that
    \begin{align*}
        \kappa_{\LLY}(x,y)=\lim_{p\to 1}\frac{1-W(\mu_x^p,\mu_y^p)}{1-p} \ge 2\min_{U\subseteq N_y}
 \left(
 \frac{|N_{H_{xy}}(U)|+1}{d_x}-\frac{|U|}{d_y}
 \right).
    \end{align*}

 \begin{case}
     $d_x< d_y$.
 \end{case}

Since $|H_x|< |H_y|$, there exists an injection $\psi:Q\rightarrow P$. 
By the definition of $\psi$, the original vertex $u$ of $u_0\in Q$ has distance $3$ from the original vertex $v$ of $\psi(u_0)$.
Let $g_{uv}$ be the number of pairs $(u_0,\psi(u_0))$ for $u_0\in Q$.
Hence $c_{uv}\cdot g_{uv}=0$.
For any $\frac{1}{1+d_x}\leq p<1$, let $\pi'_p: V\times V\to [0,1]$ be the map defined as follows:
    \begin{center}
    $\pi'_p(u,v)=\begin{cases}
    p-\frac{1-p}{d_x}, &{\rm if}\ u=x, v=y;\\
    \frac{1-p}{d_y}, &{\rm if}\ u=v=x;\\
    \frac{1-p}{d_x}, &{\rm if}\ u=v=y;\\
    (c_{uv}+g_{uv})\cdot\frac{1-p}{d_xd_y}, &{\rm if}\ u\in N_x,v\in N_y;\\
    (d_x-\sum\limits_{w\in N_x} (c_{wv}+g_{wv}))\cdot\frac{1-p}{d_xd_y}, &{\rm if}\ u=x,v\in N_y;\\
    0, &{\rm otherwise}.
    \end{cases}$
    \end{center}
Clearly, $\pi'_p$ is a transport plan.
Recall that $|M|+|Q|=|H_x|=(d_x-1)d_y$.
Therefore,
\begin{align*}
    W(\mu_x^p,\mu_y^p)&\leq  \sum_{u,v\in V}\rho(u,v)\pi'_p(u,v)\\
   & =\pi'_p(x,y)+\sum_{\substack{u\in N_x,v\in N_y\\ c_{uv}>0}}\pi'_p(u,v)+2\sum_{v\in N_y}\pi'_p(x,v)+3\sum_{\substack{u\in N_x, v\in N_y\\g_{uv}>0}}\pi'_p(u,v)\\
   &=\left(p-\frac{1-p}{d_x}\right)+\left(|M|\cdot \frac{1-p}{d_xd_y}\right)+2\left(\frac{1-p}{d_x}-\frac{1-p}{d_y}\right)+3\left(|Q|\cdot\frac{1-p}{d_xd_y}\right)\\
   &=p+\frac{1-p}{d_x}-2\cdot\frac{1-p}{d_y}+(d_x-1)d_y\cdot\frac{1-p}{d_xd_y}+2|Q|\cdot\frac{1-p}{d_xd_y}\\
   &\leq p+\frac{1-p}{d_x}-2\cdot\frac{1-p}{d_y}+(d_x-1)d_y\cdot\frac{1-p}{d_xd_y}+2(d_x-d_y+D)\cdot\frac{1-p}{d_xd_y}\\
   &=1-2\cdot\frac{1-p}{d_y}+2(d_x-d_y+D)\cdot\frac{1-p}{d_xd_y}.
\end{align*}
It follows that
    \begin{align*}
        \kappa_{\LLY}(x,y)=\lim_{p\to 1}\frac{1-W(\mu_x^p,\mu_y^p)}{1-p} \ge 2\min_{U\subseteq N_y}
 \left(
 \frac{|N_{H_{xy}}(U)|+1}{d_x}-\frac{|U|}{d_y}
 \right).
    \end{align*}

The proof is complete.
\end{proof}

\begin{corollary}\label{bipartite-curvature-lowerbound}
    Let $G$ be a bipartite graph.
    Let $xy$ be an edge in $E(G)$ such that $d_x\geq d_y$.
    Define $H\coloneqq G[N_x,N_y]$.
    If $|N_H(S)|/d_x\ge |S|/ d_y$ for every subset $S\subseteq N_y$, then $$\kappa_{\LLY}(x,y)=\frac{2}{d_x}.$$
\end{corollary}
\begin{proof}
    Combined with~\cref{lem:hall-formula} and~\cref{LLYupperbound}, the corollary holds immediately.
\end{proof}

\begin{proof}[\bf{{Proof of~\cref{thm:main}}}]
    Choose an edge $xy\in E$.
    Without loss of generality, assume that $d_x\ge d_y$ (we only use the fact $m\leq n$ for the case $\delta(G)=m$).  
Consider the local bipartite graph $H\coloneqq G[N_x,N_y]$.
Our goal is to prove that for every edge $xy$ such that $d_x\ge d_y$ and for every subset $S\subseteq N_y$,
\[
   \frac{|N_H(S)|}{d_x}\ge \frac{|S|}{d_y}.
\]

If $\delta(G)=n$, then every vertex in $X$ is adjacent to all vertices of $Y$, and hence $G\cong K_{m,n}$.  
In this case $H\cong K_{m,n}$.
Thus the estimate holds immediately since $|N_H(S)|=d_x-1$ for any non-empty set $S\subseteq N_y$  while $|S|\le d_y-1$ and $d_x\ge d_y$.

It remains to consider the case
$ \delta\coloneqq\delta(G)\ge \floor{\frac{m+n}{3}}+1$.
Let $P_x$ be the part containing $x$, and let $P_y$ be the other part containing $y$.  
Thus $|P_x|+|P_y|=m+n$.  
Since $\delta$ is an integer and $\delta>(m+n)/3$, we have $ 3\delta-(m+n)-1\ge0$.
Take $S\subseteq N_y$.

First suppose $|S| >|P_x|-\delta$. 
As every vertex $u\in N_x$ lies in $P_y$, each $u$ has at most $|P_x|-\delta$ non-neighbors in $P_x$.  
Since $|S| >|P_x|-\delta$, every $u\in N_x$ is adjacent to at least one vertex of $S$.  
Hence $N_H(S)=N_x$.
It follows that
\[
   \frac{|N_H(S)|}{d_x}=\frac{d_x-1}{d_x}
   \ge \frac{d_y-1}{d_y}
   \ge \frac{|S|}{d_y}.
\]

Then suppose $|S|\le |P_x|-\delta$.  
If $S=\emptyset$, there is nothing to prove.  For $S\neq \emptyset$,  choose a vertex $v\in S$.  
Note that  $v$ lies in $P_x$.  
Thus, 
\[
   |N_H(v)|
   \ge d_v-\bigl(|P_y|-(d_x-1)\bigr)
   \ge \delta-|P_y|+d_x-1,
\]
where $|P_y|-(d_x-1)$ denotes the size of $P_y\setminus N_x$.
Therefore $|N_H(S)|\ge \delta-|P_y|+d_x-1$.  
Since $|S|\le |P_x|-\delta$, it suffices to show
\[
   \frac{\delta-|P_y|+d_x-1}{d_x}\ge \frac{|P_x|-\delta}{d_y}.
\]
Indeed, 
\begin{align*}
    d_y(\delta-|P_y|+d_x-1)-d_x(|P_x|-\delta)& = d_x(d_y-|P_x|+\delta)+d_y(\delta-|P_y|-1)\\
    &\geq d_y(d_y+2\delta-|P_x|-|P_y|-1)\\
    &\geq d_y(3\delta-m-n-1)\geq 0.
\end{align*}

Overall, we have 
\[
   \frac{|N_H(S)|}{d_x}\ge \frac{|S|}{d_y}.
\]
for every subset $S\subseteq N_y$.
The result follows immediately by~\cref{bipartite-curvature-lowerbound}.
\end{proof}
 
Now we discuss the sharpness.

\begin{theorem}
\label{sharpness}
Let $m\ge n\ge2$.
Let $r\coloneqq \min\left\{n,\left\lfloor\frac{m+n}{3}\right\rfloor+1\right\}-1$.
Then there exists a bipartite graph $G_0=(X,Y;E)$ with $|X|=m$, $|Y|=n$, and $\delta(G_0)=r$, containing an edge of non-positive Lin--Lu--Yau curvature.
\end{theorem}

\begin{proof}
First, we suppose that $n\geq 3$.
Note that $r\le n-1$ and $3r\le m+n$.
Choose an integer $i$ such that $\max\{1,2r-n\}\le i\le \min\{r-1,m-r\}$, and let $j=r-i$.
Partition $X=\{x\}\sqcup I\sqcup P\sqcup R$
with $|I|=i$, $|P|=j-1$, and $|R|=m-r$, and partition $Y=\{y\}\sqcup J\sqcup Q\sqcup S$
 with $|J|=j$, $|Q|=i-1$, and $|S|=n-r$.
The edge set is defined as $E\coloneqq\{xy\}\sqcup E_1\sqcup E_2\sqcup E_3\sqcup E_4$, where $E_1\coloneqq\{xv\ |\  v\in J\cup Q\}$, $E_2\coloneqq\{yv\ |\ v\in I\cup P\}$, $E_3\coloneqq \{ uv\ |\ u\in  I\cup P\cup R,\ v\in  Q\cup S\}$, and $E_4\coloneqq \{ uv\ |\ u\in  P\cup R,\ v\in J\}$.
A schematic of $G_0$ is depicted in Figure~\ref{G}.
It is direct to check that $\delta(G_0)=r$.
Consider the function $f:N(x)\cup N(y)\rightarrow\mathbb{Z}$ given by
$$f(z)= 
\begin{cases}
-1, & \text { if } z\in J;\\
0, & \text { if } z\in \{x\} \cup P; \\
1, & \text { if } z\in \{y\}\cup Q;\\
2, & \text { if } z\in I.
\end{cases}$$
Then, $f(y)-f(x)=1$ and $f\in Lip(1)$.
By~\cref{Curvature via the Laplacian}, we have 
\[\kappa_{\LLY}(x,y)\leq \frac{1}{r}\left(-|J|+|Q|+1-0\right)-\frac{1}{r}\left(2|I|+0-r\right)= 0.\]

For $n=2$, let $Y=\{y_1,y_2\}$.
Choose distinct $u,v\in X$.
Join $u$ to both $y_1$ and $y_2$, and join $v$ only to $y_2$. 
Join every remaining vertex of $X$ only to $y_1$.  Then $\delta(G_0)=r=1$ and $d_u=d_{y_2}=2$.  
It follows from~\cref{no C3 C4 C5} that
\[
\kappa_{\LLY}(u,y_2)=\frac{2}{2}+\frac{2}{2}-2=0.
\]
The proof is complete.
\end{proof}

\begin{figure}[htbp]
    \centering    \begin{tikzpicture}[
    block/.style={draw, rounded corners=2.2pt, minimum width=1.15cm, minimum height=0.74cm, inner sep=2pt, font=\Large},
    vertex/.style={circle, fill=black, inner sep=2.1pt},
    header/.style={font=\Large},
    ed/.style={line width=0.75pt, line cap=round},
    noed/.style={line width=0.60pt, dashed, line cap=round},
]

\def\L{0}
\def\R{8.0}
\node[header] at (\L,5.65) {$X$};
\node[header] at (\R,5.65) {$Y$};

\node[vertex,label=left:{\Large $x$}] (x) at (\L,5.0) {};
\node[block] (I) at (\L,3.55) {$I$};
\node[block] (P) at (\L,2.10) {$P$};
\node[block] (R) at (\L,0.65) {$R$};

\node[vertex,label=right:{\Large $y$}] (y) at (\R,5.0) {};
\node[block] (J) at (\R,3.55) {$J$};
\node[block] (Q) at (\R,2.10) {$Q$};
\node[block] (S) at (\R,0.65) {$S$};

\draw[ed, ultra thick] (I.east) -- (Q.west);
\draw[ed, ultra thick] (I.east) -- (S.west);
\draw[ed, ultra thick] (P.east) -- (Q.west);
\draw[ed, ultra thick] (P.east) -- (S.west);
\draw[ed, ultra thick] (R.east) -- (Q.west);
\draw[ed, ultra thick] (R.east) -- (S.west);
\draw[ed, ultra thick] (P.east) -- (J.west);
\draw[ed, ultra thick] (R.east) -- (J.west);

\draw[ed] (x) -- (y);          
\draw[ed, ultra thick] (x) -- (J.west);     
\draw[ed, ultra thick] (x) -- (Q.west);     
\draw[ed, ultra thick] (y) -- (I.east);     
\draw[ed, ultra thick] (y) -- (P.east);     

\end{tikzpicture}
    \caption{A schematic of the graph $G_0$.}
    \label{G}
\end{figure}
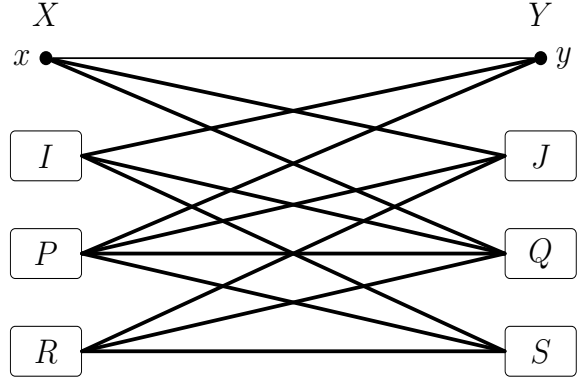

Based on~\cref{sharpness}, we prove~\cref{connectivity}.

\begin{proof}[Proof of~\cref{connectivity}]
By~\eqref{whitney}, we have $\delta(G)\geq k(G)\geq \cR(m,n)$ (or $\delta(G)\geq k'(G)\geq \cR(m,n)$.
Hence $G$ has positive Lin-Lu-Yau curvature by~\cref{thm:main}.
To illustrate the threshold is sharp, we only need to prove that $G_0$ satisfies $\delta(G_0)=k(G_0)$.

It remains to show that no set of fewer than \(r\) vertices disconnects \(G_0\).
Let \(U\subseteq V(G_0)\) with \(|S|<r\).

For $n=1$, there is nothing to prove.
For $n=2$, $G_0$ is a tree.
Hence $\delta(G_0)=k(G_0)$.

For $n\geq 3$, if \(P\setminus U\neq\emptyset\), choose \(p\in P\setminus U\). 
Since every vertex of \(P\) is adjacent to every vertex of \(Y\), all vertices of \(Y\setminus U\) lie in the component of \(p\). 
Moreover, as \(d(z)\ge r>|U|\), every vertex \(z\in X\setminus U\) has a neighbor in \(Y\setminus U\).
Hence \(z\) also belongs to the component of \(p\). 
Therefore \(G-U\) is connected.
Similarly, if \(Q\setminus U\neq\emptyset\), then \(G-U\) is connected.
It remains to consider $P\setminus U=Q\setminus U=\emptyset$, which implies $P\subseteq T$ and $ Q\subseteq T$.
Since $|P|+|Q|=r-2$ and \(|U|<r\), at most one vertex outside \(P\cup Q\) is deleted. 
Observe that \(G-(P\cup Q)\) contains the blow-up of the cycle $x-y-I-S-R-J-x$.
Therefore, removing at most one additional vertex does not break the connectedness of \(G-(P\cup Q)\).
Thus \(G-U\) is still connected.

Overall, every vertex cut of $G_0$ has size at least $r$, which means $k(G_0)\geq \delta(G_0)= r$.
Combined with~\eqref{whitney}, we obtain
$k(G_0)=k'(G_0)=\delta(G_0)=r$.
\end{proof}

\section{Curvature extremal problem on bipartite graphs}\label{sec:extremal}

In this section, we consider the extremal bipartite graphs.

\begin{proof}[Proof of~\cref{main}]
Suppose that there are at most $D(m,n)-1$ edges in the complement graph $\overline{G}$ of $G=(X, Y;E)$.
Assume that there is an edge $xy$ in $G$ such that  $\kappa_{\LLY}(x,y)\leq 0$. 
By~\cref{no C3 C4 C5}, we know $d_x\geq 2$ and $d_y\geq 2$.
 Without loss of generality, let $x\in X$ and $y\in Y$.
 Define $H\coloneqq G[N_x,N_y]$.
 By~\cref{lem:hall-formula}, there exists $S\subseteq N_y$ such that $(|N_{H}(S)|+1)d_y\leq |S|d_x$.
 We count the number of edges in $\overline{G}$.
 Observe that there is no edge between $S$ and $N_x\setminus N_H(S)$ in $G$.
 Let $|N_x\setminus N_H(S)|=r$ and $|S|=s$.
  $(|N_{H}(S)|+1)d_y\leq |S|d_x$ implies that $r\geq\frac{(d_y-s)d_x}{d_y}$.
 Therefore, 
 \begin{align*}
     e(\overline{G})&\geq rs+n+m-d_x-d_y\\
     &\geq  \frac{(d_y-s)s\cdot d_x}{d_y}+n+m-d_x-d_y\\
     &\geq \frac{(d_y-1)\cdot d_x}{d_y}+n+m-d_x-d_y\\
     &=n+m-d_y-\frac{d_x}{d_y}.
 \end{align*}
 Since $2\leq d_y\leq m$ and $2\leq d_x \leq n$, $e(\overline{G})> D(m,n)-1$, a contradiction.
 Therefore, for any edge $xy$, $\kappa_{\LLY}(x,y)>0$.

 Now we consider the sharpness.
 For $n\geq 2m$,
let $G_1=(X, Y;E)$ be the bipartite graph  with $|X|=m$ and $|Y|=n$ defined as follows.
The vertex sets $X\coloneqq\{x,w\}\sqcup J$ and $Y\coloneqq \{y\}\sqcup P\sqcup Q$, where $|Q|=\ceil{\frac{n}{2}}$.
The edge set 
$E\coloneqq  E_1\sqcup E_2\sqcup E_3,$ where $E_1\coloneqq\{xv\ |\ v\in Y\}$, $E_2\coloneqq\{wv\ |\ v\in \{y\}\cup P\}$, and $E_3\coloneqq\{uv\ |\ u\in J, v\in  P\cup Q\}$.
Then, $G_1$ has exactly $mn-D(m,n)$ edges.
A schematic of $G_1$ is depicted in~\cref{G_1}.
\begin{figure}[htbp]
    \centering    \begin{tikzpicture}[
    line cap=round,
    line join=round,
    edge/.style={
        draw=black,
        line width=1pt
    },
    vertex/.style={
        circle,
        fill=black,
        inner sep=2.2pt
    },
    part/.style={
        rounded rectangle,
        rounded rectangle arc length=180,
        draw=black,
        fill=white,
        line width=1pt,
        minimum width=2.8cm,
        minimum height=1.25cm,
        inner sep=0pt,
        outer sep=0pt,
        font=\Large
    }
]

\coordinate (x) at (0,3);
\coordinate (y) at (0,0);
\coordinate (w) at (2.65,3);

\node[part] (P) at (3.75,0) {$P$};
\node[part] (J) at (6.35,3) {$J$};
\node[part] (Q) at (8.35,0) {$Q$};

\draw[edge] (x) -- (y);
\draw[edge] (y) -- (w);

\draw[edge] (x) -- (P);
\draw[edge] (w) -- (P);

\draw[edge] (x) -- (Q);

\draw[edge] (P) -- (J);
\draw[edge] (J) -- (Q);

\node[
    vertex,
    label={[font=\Large]above left:$x$}
] at (x) {};

\node[
    vertex,
    label={[font=\Large]below left:$y$}
] at (y) {};

\node[
    vertex,
    label={[font=\Large]above:$w$}
] at (w) {};

\end{tikzpicture}
    \caption{A schematic of the graph $G_1$.}
    \label{G_1}
\end{figure}
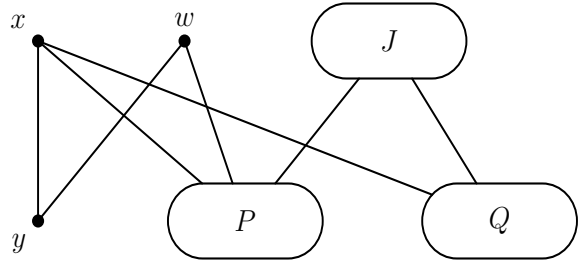

By~\cref{lem:hall-formula},
\begin{align*}
 \kappa_{\LLY}(x,y)\leq 2\left(\frac{|N_{P\cup Q}(w)|+1}{d_x}-\frac{1}{d_y}\right)=2\left( \floor{\frac{n}{2}}\cdot\frac{1}{n}-\frac{1}{2}\right)\leq 0.   
\end{align*}
 
 For $m\leq n<2m$, let $G_2=(X, Y;E)$ be the bipartite graph  with $|X|=m$ and $|Y|=n$ defined as follows.
The vertex sets $X\coloneqq\{x,w\}\sqcup A$ and $Y\coloneqq \{y\}\sqcup B$.
The edge set 
$E\coloneqq  E_1\sqcup E_2,$ where $E_1\coloneqq\{uy\ |\ u\in X\}$ and $E_2\coloneqq\{uv\ |\ u\in A\cup\{x\},v\in B\}$.
Then, $G_2$ has exactly $mn-D(m,n)$ edges.
A schematic of $G_2$ is depicted in~\cref{G_2}.
\begin{figure}[htbp]
    \centering    \begin{tikzpicture}[
    line cap=round,
    line join=round,
    edge/.style={
        draw=black,
        line width=1pt
    },
    vertex/.style={
        circle,
        fill=black,
        inner sep=2.2pt
    },
  part/.style={
        rounded rectangle,
        rounded rectangle arc length=180,
        draw=black,
        fill=white,
        line width=1pt,
        minimum width=2.6cm,
        minimum height=1.35cm,
        inner sep=0pt,
        outer sep=0pt,
        font=\Large
    }
]

\coordinate (x) at (0,3);
\coordinate (y) at (0,0);
\coordinate (w) at (2.7,3);

\node[part] (A) at (6.2,3) {$A$};
\node[part] (B) at (6.2,0) {$B$};

\draw[edge] (x) -- (y);
\draw[edge] (y) -- (w);
\draw[edge] (y) -- (A);
\draw[edge] (x) -- (B);
\draw[edge] (A) -- (B);

\node[
    vertex,
    label={[font=\Large]left:$x$}
] at (x) {};

\node[
    vertex,
    label={[font=\Large]left:$y$}
] at (y) {};

\node[
    vertex,
    label={[font=\Large]above:$w$}
] at (w) {};

\end{tikzpicture}
    \caption{A schematic of the graph $G_2$.}
    \label{G_2}
\end{figure}
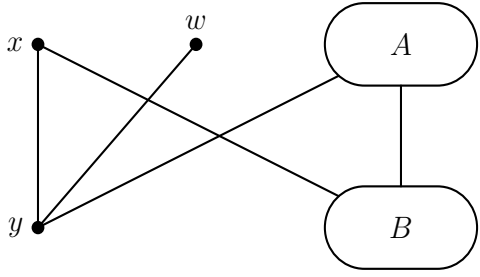

By~\cref{lem:hall-formula},
\begin{align*}
 \kappa_{\LLY}(x,y)\leq 2\left(\frac{|N_{B}(w)|+1}{d_x}-\frac{1}{d_y}\right) =2\left(\frac{1}{n}-\frac{1}{m}\right)  \leq 0.
\end{align*}
The proof is complete.
\end{proof}

\section{Random bipartite graphs are positively curved}\label{sec:random_bipartite}

In this section, we prove the second result.

Let $G$ be a random bipartite graph $B(m,n,p)$ with parts $X$ and $Y$.  
Denote by $N = m + n$ and $s = \min(m,n)$.
Denote by $U = N_x$ and $V = N_y$. 
Let $H = G[U, V]$ be the subgraph on $U$ and $V$.

\begin{lemma}[{Chernoff bound~\cite{chernoff1981NoteInequalityInvolving} or~\cite[Lemma 5.2]{LLY11}}]\label{lem:Chernoff}
    Let $X_1, \ldots, X_n$ be independent random variables with
    \begin{equation*}
        \Pr\paren{X_i=1}=p_i, \quad \Pr\paren{X_i=0}=1-p_i.
    \end{equation*}
    The expectation of the sum $X = \sum_{i=1}^n X_i$ is $\Ex[X]=\sum_{i=1}^n p_i$. 
    Then
    \begin{align*}
        \Pr(X \leq \Ex[X] - \lambda) &\leq \exp\paren{-\lambda^2 / 2 \Ex[X]}, \\
        \Pr(X \geq \Ex[X] + \lambda) &\leq \exp\paren{-\lambda^2 /(2 \Ex[X]+2 \lambda / 3)}.
    \end{align*}
\end{lemma}

\begin{lemma}\label{lem:weighted_hall_failure}
    Let $H$ be a random bipartite graph $B(a,b, p)$ with $1 \leq a \leq b \leq N$. 
    For every $K > 0$, there exists $C_K > 0$ such that $p a \geq C_K \log N$ implies that
    \begin{align*}
        \Pr\paren*{\exists S \subseteq X: \card{N_H(S)} < \frac{b+1}{a+1} \card{S}} \leq N^{-K}. 
    \end{align*}
    In particular, suppose $xy \in E(G)$, then by~\cref{bipartite-curvature-lowerbound}, we have $p a \geq C_K \log N$ implies that
    \begin{align*}
        \Pr\paren*{\kappa_{\LLY}(x,y) \neq \frac{2}{\max(d_x, d_y)} \mid xy \in E(G)} \leq N^{-K}. 
    \end{align*}
\end{lemma}

\begin{proof}
    Fix $S \subseteq X$, and denote by $\card{S} = r$. 
    If $|N_H(S)| < \frac{b+1}{a+1}r$, then $|N_H(S)| \leq \ceil{\frac{b+1}{a+1}r} - 1$. 
    Therefore $Y \setminus N_H(S)$ contains a set $T$ of size $t_r = \floor{\frac{b+1}{a+1}(a+1-r)}$. 
    There are no edges between $S$ and $T$, hence
    \begin{align*}
        \Pr &\leq \sum_{r=1}^a \binom{a}{r} \binom{b}{t_r} (1-p)^{r t_r} \\
            &\leq \sum_{r=1}^a \binom{a}{r} \binom{b}{t_r} e^{-p r t_r}. 
    \end{align*}  
    Take $q_r = \min(r, a+1 - r)$. 
    There exist absolute constants $c_0, C_0$ such that
    \begin{align*}
        r t_r \geq c_0 b q_r \quad \text{ and }\quad
        \log\paren*{\binom{a}{r}\binom{b}{t_r}} \leq C_0 \frac{b}{a} q_r \log \frac{ea}{q_r}. 
    \end{align*}
    Therefore
    \begin{align*}
        \Pr &\leq \sum_{r=1}^a \exp\paren*{C_0 \frac{b}{a} q_r \log\frac{ea}{q_r} - c_0 p b q_r} \\
        &\leq \sum_{r=1}^a \exp\paren*{\frac{b}{a} q_r \paren*{C_0 \log \frac{ea}{q_r} - c_0 p a}}
    \end{align*}
    Note that $\log \frac{ea}{q_r} \leq \log (eN)$. 
    Choose $C_A$ sufficiently large makes every summand at most $N^{-A-2}$. 
\end{proof}

\begin{proof}[{Proof of~\cref{thm:random_bipartite_positive}}]
    Consider $\mathcal{D} = \set{\forall z \in X \cup Y: d_z \geq \frac{1}{2} p \card{P}}$, where $P$ is the opposite vertex part of $z$. 
    By the Chernoff tail inequality, we have 
    \begin{align*}
        \Pr(\mathcal{D}^c) &\leq m e^{-np/8} + n e^{-mp/8} \leq N e^{-ps/8}. 
    \end{align*}
    Since $p^2 s \geq C \log N$ and $p \leq 1$, we have $p s \geq p^2 s \geq C \log N$. 
    Thus for sufficiently large $C$, we have $\Pr(\mathcal{D}^c) \leq N^{-4}$. 

    Fix an edge $xy$ with $x, y \in X \cup Y$. 
    Then the subgraph $H_{xy} = G[U, V]$ is a random bipartite graph from $B(|U|, |V|, p)$. 
    Suppose the degrees satisfy $d_x \geq \frac{1}{2} n p$, $d_y \geq \frac{1}{2} m p$. 
    Let $a \coloneqq \min{d_x, d_y} - 1$. 
    Then $pa \geq \frac{1}{3} p^2 s$ for sufficiently large $N$. 
    Take $K = 5$ in~\cref{lem:weighted_hall_failure}, then 
    \begin{align*}
        \Pr\paren*{\kappa_{\LLY}(x, y) \neq \frac{2}{\max(d_x, d_y)} \mid xy \in E(G)} \leq N^{-5}. 
    \end{align*}
    Note that there are at most $mn \leq N^2$ edges, so
    \begin{align*}
        \Pr\paren*{\exists xy \in E(G) : \kappa_{\LLY}(x,y) \neq \frac{2}{\max(d_x, d_y)} \text{ and } \mathcal{D}} \leq N^{-3}. 
    \end{align*}
    The theorem follows by noting that $\frac{2}{\max(d_x, d_y)} > 0$. 
\end{proof}

\section{Concluding remark}\label{sec:random_graph}


The probabilistic result holds not only for random bipartite graphs, but also for random graphs. 

Lin, Lu and Yau~\cite{LLY11} have proved that for a fixed edge, the probability that its curvature is positive tends to 1. 

\begin{theorem}[{\cite[Theorem 5.1]{LLY11}}]\label{thm:all_edge_positive}
    Let $G$ be a random graph from $G(n, p)$. 
    Let $xy$ be an edge of $G$.  
    If $p \gg \sqrt[3]{(\ln n) / n}$, then
    \begin{align}
        \lim_{n \to \infty} \Pr\paren*{\kappa_{\LLY}(x, y) = (1 + \smallo(1))p} =  1.
    \end{align}
\end{theorem}

By careful analysis, one can obtain the following uniform version which holds for all edges. 
Note that `for each edge $xy$' is inside the limit and probability. 

\begin{theorem}
    Suppose $p \gg \sqrt[3]{(\ln n) / n} $. 
    Let $G$ be a random graph from $G(n, p)$. 
    Then
    \begin{align*}
        \lim_{n \to \infty} \Pr\paren*{ G \text{ is positively curved}} = 1.
    \end{align*}
\end{theorem}

\begin{corollary}\label{thm:graph_complmentary_positive}
    Suppose $\sqrt[3]{(\ln n) / n} \ll p \ll 1 - 50 \sqrt[3]{(\ln n) / n}$. 
    Let $G$ be a random graph from $G(n, p)$. 
    Then
    \begin{align*}
        \lim_{n \to \infty} \Pr\paren*{ G \text{ and } \overline{G} \text{ are positively curved}} = 1.
    \end{align*}
\end{corollary}

Here is a sketch of proof. 
It is similar to the original proof of~\cref{thm:all_edge_positive}.
Firstly, we use Chernoff tail bound to estimate the range of vertex degrees. 
More specifically, if $p \geq (16 \ln n)/(3n)$, then with probability at least $1 - 2/n^3$, all vertex degrees $d_x$ of $G(n, p)$ fall in the range
    \begin{equation*}
        \paren*{(n-1) p - \sqrt{8 n p \ln n}, (n-1) p + \sqrt{12 n p \ln n}}.
    \end{equation*}
Secondly, we use Chernoff tail bound to estimate the range of vertex codegrees. 
More specifically, if $p \geq \sqrt{(20 \ln n) / (3 n)}$, then with probability at least $1 - 1/n^3$, all codegrees $d_{xy}$ of $G(n,p)$ fall in the range
    \begin{equation*}
        \paren*{(n-2) p^2 - \sqrt{10 n p^2 \ln n}, (n-2) p^2 + \sqrt{15 n p^2 \ln n}}.
    \end{equation*}
Thirdly, we have a formula for the lower bound of the Lin--Lu--Yau curvature~\cite[Lemma 5.5]{LLY11} by 
    \begin{align*}
        \kappa_{\LLY}(x, y) \geq 1 - \frac{1}{d_y} \sum_{u \in N_x} \rho(u, \phi(u))+\frac{1}{d_x}-\frac{3\paren{d_y - d_x}}{d_y},
    \end{align*}
where $\phi: N_x \to N_y$ is an injection. 
Then we show that with high probability, there is a large matching between $N_x$ and $N_y$. 
In other words, a lot of the distances $\rho(u, \phi(u))$ is $1$. 
Lastly, we bound the positive and negative terms in the formula by above estimates.

\end{document}